\documentclass[reqno,11pt]{amsart}
\usepackage{epsfig,amscd,amssymb,amsmath,amsfonts}
\usepackage{lscape}
\usepackage{amsmath}
\usepackage{graphicx}
\usepackage{amsthm,color}
\usepackage{tikz}
\usepackage{float}
\usepackage{xparse}

\usetikzlibrary{calc}
\usetikzlibrary{shapes.geometric}
\usetikzlibrary{shapes,positioning,intersections,quotes}
\usetikzlibrary{graphs}
\usetikzlibrary{graphs,quotes}
\usetikzlibrary{decorations.pathmorphing}
\usetikzlibrary{decorations.markings}

\tikzset{snake it/.style={decorate, decoration=snake}}
\tikzset{snake it/.style={decorate, decoration=snake}}

\usetikzlibrary{decorations.pathreplacing,decorations.markings,snakes}
\usetikzlibrary{backgrounds}

\newtheorem{theorem}{Theorem}[section]
\newtheorem{lemma}[theorem]{Lemma}

\theoremstyle{definition}
\newtheorem{definition}[theorem]{Definition}
\newtheorem{corollary}[theorem]{Corollary}

\newtheorem{example}[theorem]{Example}

\theoremstyle{remark}
\newtheorem{remark}[theorem]{Remark}

\numberwithin{equation}{section}

\usepackage[margin=1.05in]{geometry}
\usepackage[colorlinks]{hyperref}
\usepackage{siunitx}

\begin{document}

\title[An acyclic $d$-partition of the $r$-uniform complete hypergraph $K_{rd}^{(r)}$]{An acyclic $d$-partition of the $r$-uniform complete hypergraph $K_{rd}^{(r)}$}


\author{Ayako Carter}
\address{Department of Mathematics and Statistics, Bowling Green State University, Bowling Green, OH 43403 }

\email{ayakoc@bgsu.edu}


\author{Eric Montoya}
\address{Department of Mathematics and Statistics, Bowling Green State University, Bowling Green, OH 43403 }

\email{emontoy@bgsu.edu}


\author{Mihai D. Staic}
\address{Department of Mathematics and Statistics, Bowling Green State University, Bowling Green, OH 43403 }
\address{Institute of Mathematics of the Romanian Academy, PO.BOX 1-764, RO-70700 Bu\-cha\-rest, Romania.}

\email{mstaic@bgsu.edu}




\begin{abstract} In this paper we introduce a $d$-partition $\mathcal{E}_d^{(r)}=(\Omega_1^{(r,d)}, \Omega_2^{(r,d)},\dots, \Omega_d^{(r,d)})$ of the $r$-uniform  complete hypergraph $K_{rd}^{(r)}$. We prove that $\mathcal{E}_d^{(r)}$ is homogeneous and  that each hypergraph $\Omega_i^{(r,d)}$ is acyclic (i.e. has zero Betti numbers). As an application, we show that the map $det^{S^r}$ is nontrivial for every $r$, which gives a partial answer to a conjecture from \cite{ls2}.

\end{abstract}
\maketitle



%

\section{Introduction}

It is well known that the complete graph $K_{2d}$ can be partitioned into $d$ spanning trees. Moreover,  one can choose those spanning trees  to be isomorphic to the twin-star on $2d$ vertices. In this paper we study an analogous problem for the $r$-uniform complete hypergraph $K_{rd}^{(r)}$.

Partitions of graphs and hypergraphs have been intensively studied over the years (for some examples see \cite{an,frc,ckv,ct,cp,lmt,sch}).   
When generalizing spanning tree partitions from graph to hypergraphs, one issue is finding the proper definition of a tree in the context of hypergraphs.  For some accounts on this problem one can check \cite{geo,gk,rrs,seb}. For our purposes, the right approach is suggested by Kalai's  results on acyclic simplicial complexes \cite{ka}. Having that paper as a model, one can define  an acyclic hypergraph as an $r$-uniform  hypergraph that contains all possible $(r-1)$-faces, ${\displaystyle \binom{rd-1}{r-1}}$ hyperedges, and its Betti numbers are all equal to zero.  Hence, an acyclic partition is a partition in which all the component hypergraphs are acyclic (see Section \ref{section2} for the precise definition). These acyclic partitions where first considered in \cite{ls2} in relation to a certain determinant-like map $det^{S^r}$. The main result of this paper is proving that for every $r$ and $d$ there exists an acyclic $d$-partition $\mathcal{E}_d^{(r)}$  of the $r$-uniform complete hypergraph $K_{rd}^{(r)}$.

To make the paper self-contained, in Section \ref{section2} we recall a few definitions and properties of hypergraphs, partitions of hypergraphs, simplicial complexes, and Betti numbers. We also briefly recall the connection between $d$-partitions of the $r$-uniform complete hypergraph $K_{rd}^{(r)}$ and the map $det^{S^r}$. 

In Section \ref{section3} we introduce $\mathcal{E}_d^{(r)}=(\Omega_1^{(r,d)}, \Omega_2^{(r,d)},\dots, \Omega_d^{(r,d)})$, a $d$-partition of the $r$-uniform complete hypergraph $K_{rd}^{(r)}$. When $r=2$ one recovers the 
well-known partition of the complete graph $K_{2d}$ as a union of twin-star graphs. We show that $\mathcal{E}_d^{(r)}$ is homogeneous, and that $\Omega_i^{(r,d)}$ and $\Omega_j^{(r,d)}$ are isomorphic for all $1\leq i<j\leq d$. We also introduce a family of $r$-uniform hypergraphs  $\Gamma_{k,r}(j_{k+1},\dots,j_r)$, and show that $\Omega_1^{(r,d)}$ can be written as their union.  

In Section \ref{section4} we introduce the leaf-equivalence relation among hypergraphs, and show that $\Gamma_{k,r}(j_{k+1},\dots,j_r)$ is leaf-equivalent to the empty $r$-uniform hypergraph.  In Section \ref{section5} we prove Theorem \ref{Th17r}, the main result of this paper, which states that the $d$-partition $\mathcal{E}_d^{(r)}$ is acyclic.  As an application, we show that the map $det^{S^r}$ is nontrivial for all $r$ and $d$, which gives a partial answer to a conjecture from \cite{ls2}.

\section{Preliminary}

\label{section2}

\subsection{Partitions of Hypergraphs}
We recall from \cite{bretto,ka,ls2} a few definitions and results about hypergraphs, partitions, simplicial complexes, and Betti numbers.   
\begin{definition} 
(1) A hypergraph $\mathcal{H}=(V, E)$ consists of two finite sets  $V = \{v_1, v_2, \dots , v_n\}$ called  the set of vertices, and  $E=\{E_1, E_2, ... , E_m\}$  a family of subsets of $V$ called the hyperedges of $\mathcal{H}$. \\
(2) If $\mathcal{H}=(V, E)$ is a hypergraph such that every hyperedge of $\mathcal{H}$ has cardinality $r$, then $\mathcal{H}$ is called an $r$-uniform hypergraph.\\
(3) For $2 \leq r \leq n$, we define the $r$-uniform complete hypergraph to be the hypergraph $K_{n}^{(r)}= (V, E)$ for which $V=\{1,2,\dots ,n\}$, and  $E$ is the family of all $r$-element subsets of $V$.

\end{definition} 

\begin{remark}
A $2$-uniform hypergraph is nothing else but a graph, and $K_n^{(2)}$ is the complete graph $K_n$.
\end{remark}

\begin{remark} To prove certain results, it is convenient to order the elements of the set $V$, and denote a hyperedge $\sigma=\{v_1,v_2,\dots,v_r\}$  by $(v_1,v_2,\dots,v_r)$ where $v_1<v_2<\dots<v_r$. We will use both notations as needed. Since we do not distinguish among permutations of the set $\sigma$, this should not create any confusion. 
\end{remark}

\begin{definition}  A hyperedge $d$-partition of $K_{rd}^{(r)}$ is an ordered collection $\mathcal{P}=(\mathcal{H}_1,\mathcal{H}_2,...,\mathcal{H}_d)$ of sub-hypergraphs $\mathcal{H}_i$ of  $K_{rd}^{(r)}$ such that:\\
(i) ${\displaystyle V(\mathcal{H}_i)=V(K_{rd}^{(r)})}$ for all $1\leq i \leq d$,  \\
(ii) ${\displaystyle E(\mathcal{H}_i)\bigcap E(\mathcal{H}_j)=\emptyset}$ for all $1\leq i\neq j\leq d$, \\
(iii) ${\displaystyle \bigcup_{i=1}^dE(\mathcal{H}_i)=E(K_{rd}^{(r)})}$. 
\end{definition}

\begin{example} For examples of partitions of the complete graph $K_{2d}=K_{2d}^{(2)}$ one can check \cite{fls,lss}.  In Figure \ref{fig1} we have a $2$-partition 
of the $3$-uniform complete hypergraph $K_{6}^{(3)}$. See also Example \ref{exampleG23} from Section \ref{section3} where this partition is denoted $(\Omega_{1}^{(3,2)},\Omega_{2}^{(3,2)})$.  

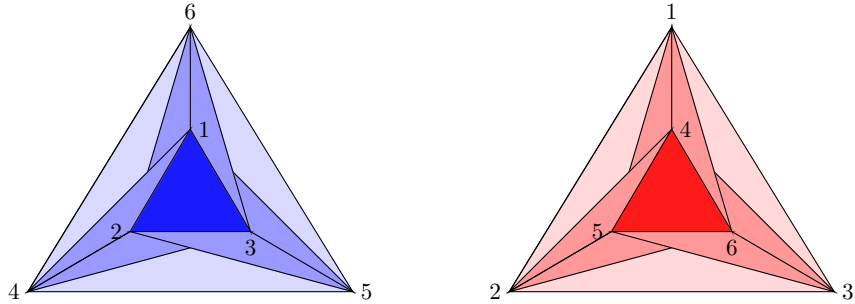
\begin{figure}[h!]
 \centering
 \begin{tikzpicture}[scale=0.8,every node/.style={scale=0.8}]
 %
 \draw[fill, color=blue!90]
(1,1.7) -- (2,0) -- (0,0) -- cycle;

\draw[fill, color=blue!15]
(3.7,-1) -- (2,0) -- (1,3.4) -- cycle;
\draw[fill, color=blue!40]
(1,1.7) -- (2,0) -- (3.7,-1) -- cycle;

\draw (1,1.7) node[below] {} -- (2,0)
node[left] {} -- (3.7,-1)
node[left] {} -- cycle;

\draw[fill, color=blue!40]
(1,1.7) -- (2,0) -- (1,3.4) -- cycle;

\draw[fill, color=blue!15]
(1,3.4) -- (0,0) -- (-1.7,-1) -- cycle;
\draw[fill, color=blue!40]
(1,3.4) -- (0,0) -- (1,1.7) -- cycle;

\draw (-1.7,-1) node[below] {} -- (0,0)
node[left] {} -- (1,3.4)
node[left] {} -- cycle;

\draw[fill, color=blue!40]
(-1.7,-1) -- (0,0) -- (1,1.7) -- cycle;

\draw[fill, color=blue!15]
(-1.7,-1) -- (2,0) -- (3.7,-1) -- cycle;

\draw[fill, color=blue!40]
(0,0) -- (2,0) -- (-1.7,-1) -- cycle;

\draw (0,0) node[below] {} -- (2,0)
node[left] {} -- (-1.7,-1)
node[left] {} -- cycle;

\draw[fill, color=blue!40]
(2,0) -- (0,0) -- (3.7,-1) -- cycle;

	\draw (0,0) node[below] {} -- (2,0)
node[below] {$3$} -- (3.7,-1)
node[right] {$5$} -- cycle;

	\draw (1,1.7) node[below] {} -- (2,0)
node[right] {} -- (1,3.4)
node[above] {$6$} -- cycle;

\draw (1,1.7) node[right] {$1$} -- (0,0)
node[left] {$2$} -- (-1.7,-1)
node[left] {$4$} -- cycle;

\draw (-1.7,-1) node[below] {} -- (3.7,-1)
node[right] {} -- (1,3.4)
node[left] {} -- cycle;
	

\draw[fill, color=red!15]
(11.7,-1) -- (10,0) -- (9,3.4) -- cycle;

\draw[fill, color=red!40]
(9,1.7) -- (10,0) -- (11.7,-1) -- cycle;

\draw (9,1.7) node[below] {} -- (10,0)
node[left] {} -- (11.7,-1)
node[left] {} -- cycle;

\draw[fill, color=red!15]
(9,3.4) -- (10,0) -- (6.3,-1) -- cycle;
\draw[fill, color=red!40]
(9,3.4) -- (8,0) -- (9,1.7) -- cycle;

\draw (6.3,-1) node[below] {} -- (8,0)
node[left] {} -- (9,3.4)
node[left] {} -- cycle;

\draw[fill, color=red!40]
(6.3,-1) -- (8,0) -- (9,1.7) -- cycle;

\draw[fill, color=red!15]
(6.3,-1) -- (10,0) -- (11.7,-1) -- cycle;

\draw[fill, color=red!40]
(8,0) -- (10,0) -- (6.3,-1) -- cycle;

\draw (8,0) node[below] {} -- (10,0)
node[left] {} -- (6.3,-1)
node[left] {} -- cycle;

\draw[fill, color=red!40]
(10,0) -- (8,0) -- (11.7,-1) -- cycle;

 \draw[fill, color=red!90]
(9,1.7) -- (10,0) -- (8,0) -- cycle;

\draw[fill, color=red!40]
(9,1.7) -- (10,0) -- (9,3.4) -- cycle;

	\draw (8,0) node[below] {} -- (10,0)
node[below] {$6$} -- (11.7,-1)
node[right] {$3$} -- cycle;

	\draw (9,1.7) node[below] {} -- (10,0)
node[right] {} -- (9,3.4)
node[above] {$1$} -- cycle;

\draw (9,1.7) node[right] {$4$} -- (8,0)
node[left] {$5$} -- (6.3,-1)
node[left] {$2$} -- cycle;

\draw (6.3,-1) node[below] {} -- (11.7,-1)
node[right] {} -- (9,3.4)
node[left] {} -- cycle;

 \end{tikzpicture}
 \caption{A $2$-partition of the $3$-uniform complete hypergraph $K_6^{(3)}$ \label{fig1}}
\end{figure}

\end{example}
The following definitions and results are taken from \cite{ls2}. 
\begin{definition} 
If $\mathcal{H}=(V,E)$ is an $r$-uniform  hypergraph, we denote 
\[
 E_{s}\left(\mathcal{H}\right) = 
 \left\lbrace 
 \{a_1,a_2,\dots,a_s\}\subseteq V\middle|\;
\begin{aligned}
 & \{a_1,a_2,\dots,a_s,x_{s+1},x_{s+2},\dots,x_{r}\}\in E(\mathcal{H})\\
 & ~{\rm for ~ some ~}x_{s+1},x_{s+2},\dots,x_{r}\in V(\mathcal{H})
\end{aligned}
\right\rbrace.
\]
An element in $E_{s}(\mathcal{H})$ is called an $s$-face of the hypergraph  $\mathcal{H}$.
\end{definition}

\begin{definition}
(1) We say that a $d$-partition $(\mathcal{H}_1,\mathcal{H}_2,\dots, \mathcal{H}_d)$ of $K_{rd}^{(r)}$ is pre-homogeneous if  $$\vert E_{s}(\mathcal{H}_i)\vert={\displaystyle \binom{rd}{s}}$$ for all $1\leq i\leq d$  and all $0\leq s\leq r-1$.\\
(2) We say that a pre-homogeneous  $d$-partition $(\mathcal{H}_1,\mathcal{H}_2,\dots, \mathcal{H}_d)$ of $K_{rd}^{(r)}$ is homogeneous if  for all $1\leq i\leq d$ we have $$\vert E(\mathcal{H}_i)\vert=\vert E_r(\mathcal{H}_i)\vert={\displaystyle \binom{rd-1}{r-1}=\frac{1}{d}\binom{rd}{r}}.$$
\end{definition}

\begin{remark} \label{remark28} Notice that the $d$-partition $(\mathcal{H}_1,\mathcal{H}_2,\dots, \mathcal{H}_d)$ of $K_{rd}^{(r)}$ is pre-homogeneous if 
and only if $E_{r-1}(\mathcal{H}_i)=K_{rd}^{(r-1)}$ for all $1\leq i\leq d$. Indeed, $E_{r-1}(\mathcal{H}_i)\subseteq K_{rd}^{(r-1)}$, and if our partition is pre-homogeneous we have ${\displaystyle \vert E_{r-1}(\mathcal{H}_i)\vert= \binom{rd}{r-1}=\vert E(K_{rd}^{(r-1)})\vert}$. This implies that $E_{r-1}(\mathcal{H}_i)=K_{rd}^{(r-1)}$.  
\end{remark}

Next we discuss how to associate a simplicial complex  to an $r$-uniform hypergraph. Let $\mathcal{H}=(V,E)$ be an $r$-uniform sub-hypergraph of $K_{n}^{(r)}$ (where $V$ is ordered and $|V|=n$). Take $\{l_1,l_2,\dots,l_n\}$ to be a basis for $\mathbb{R}^n$. To $\mathcal{H}$ we associate an $(r-1)$-dimensional simplicial complex  $X(\mathcal{H})\subseteq \mathbb{R}^n$, where the $(s-1)$-skeleton of $X(\mathcal{H})$  is given by  
        \[
        X_{s-1}\left(\mathcal{H}\right) = 
        \left\lbrace 
        t_1l_{a_1}+t_2l_{a_2}+\dots + t_sl_{a_s}\in \mathbb{R}^n\middle|\;
        \begin{aligned}
        & t_1,\dots, t_s \in [0,1], \; t_1 +  \dots + t_s = 1,  &\\
				& \text{and }  (a_1, a_2, \dots ,a_s) \in E_s(\mathcal{H})&
        \end{aligned}
        \right\rbrace,
        \]
for each $0\leq s\leq r$.

\begin{definition} Let  $\mathcal{H}$ be an $r$-uniform sub-hypergraph of $K_{n}^{(r)}$. We define the Betti numbers of $\mathcal{H}$  as  the  reduced Betti numbers of its corresponding simplicial complex  $X(\mathcal{H})$. 
\end{definition} 

For the convenience of the reader, and in order to establish notation that will be used later, we give a few  details on how to compute Betti numbers of a hypergraph. 

Let $\mathcal{H}=(V,E)$ be an $r$-uniform  hypergraph.  Consider the chain complex $\mathcal{K}(\mathcal{H})$
$$0\to C_{r-1}(\mathcal{H})\stackrel{\partial_{r-1}}{\rightarrow} C_{r-2}(\mathcal{H})\stackrel{\partial_{r-2}}{\rightarrow}\dots \stackrel{\partial_{1}}{\rightarrow} C_{0}(\mathcal{H})\stackrel{\partial_{0}}{\rightarrow} C_{-1}(\mathcal{H})\to 0,$$
where for $1\leq k\leq r$ we have $C_{k-1}(\mathcal{H})$ is the $\mathbb{Q}$-vector space with a basis 
$$\{[a_1,\dots,a_k]~|~(a_1,a_2,\dots, a_k)\in E_k(\mathcal{H})\},$$ and $C_{-1}(\mathcal{H})$ is the $1$-dimensional $\mathbb{Q}$-vector space with the basis $\{[\emptyset]\}$. 

The map $\partial_{k-1} : C_{k-1}(\mathcal{H})\to C_{k-2}(\mathcal{H})$ is determined by the usual formula
$$ \partial_{k-1}([a_1,a_2,\dots,a_k])=[a_2,a_3,\dots,a_k]-[a_1,a_3,\dots,a_k]+\dots+(-1)^{k+1}[a_1,a_2,\dots,a_{k-1}].$$
With this notation we have 
$$b_{k-1}(\mathcal{H})=b_{k-1}(X(\mathcal{H}))=dim_{\mathbb{Q}}(H_{k-1}(\mathcal{K}(\mathcal{H})))=dim_{\mathbb{Q}}\left(\frac{Ker(\partial_{k-1})}{Im(\partial_k)}\right),$$
for all $0\leq k\leq r$.

\begin{definition} Let $\mathcal{P}=(\mathcal{H}_1,\mathcal{H}_2,\dots, \mathcal{H}_d)$ be a homogeneous $d$-partition of $K_{rd}^{(r)}$. We say that $\mathcal{P}$ is acyclic if $b_{k-1}(\mathcal{H}_i)=0$ for each $1\leq i\leq d$ and all $0\leq k\leq r$. 
\end{definition}

\begin{remark} 
Recall from \cite{ka} that a $k$-dimensional simplicial complex $X$ over a set with $n$ elements is called $\mathbb{Q}$-acyclic  if $X$  has complete $(k-1)$-skeleton, the number of $k$-faces of $X$ is ${\displaystyle \binom{n-1}{k}}$, and $H_k(X)=0$. Also, notice that if $\mathcal{H}$ is an $r$-uniform hypergraph,  then to an $s$-face $\tau\in  E_{s}\left(\mathcal{H}\right)$ we associate an $(s-1)$-dimensional simplex $[\tau]\in X_{s-1}\left(\mathcal{H}\right)$. In particular, the number of $(s-1)$-dimensional simplexes of $X(\mathcal{H})$ is equal to $|E_{s}\left(\mathcal{H}\right)|$, and the simplicial complex $X(\mathcal{H})$ is of dimension $r-1$. 

From these remarks, one can see that a $d$-partition $(\mathcal{H}_1,\mathcal{H}_2,\dots, \mathcal{H}_d)$ of $K_{rd}^{(r)}$ is acyclic if and only if the simplicial complex $X(\mathcal{H}_i)$ is $\mathbb{Q}$-acyclic for each $1\leq i\leq d$. 
\end{remark}
\begin{remark}
At the time they wrote their paper, the authors of \cite{ls2} were not aware of \cite{ka}. This explains the unfortunate shift in the notation (i.e. a $s$-face of $\mathcal{H}$ corresponds to a $(s-1)$ simplex in $X(\mathcal{H})$). 
\end{remark}

\begin{remark} \label{rem211} Let  $(\mathcal{H}_1,\mathcal{H}_2,\dots, \mathcal{H}_d)$ be a homogeneous  $d$-partition  of $K_{rd}^{(r)}$. If $\mathcal{H}$ is one of the hypergraphs $\mathcal{H}_i$, then the Euler characteristic of the chain complex $\mathcal{K}(\mathcal{H})$ is given by 
$$\chi(\mathcal{K}(\mathcal{H}))=\binom{rd}{0}-\binom{rd}{1}+\binom{rd}{2}-\dots +(-1)^{r-1}\binom{rd}{r-1}+(-1)^{r}\frac{1}{d}\binom{rd}{r}=0.$$
In particular this implies that 
$$b_{-1}(\mathcal{H})-b_0(\mathcal{H})+b_1(\mathcal{H})-\dots+(-1)^{r-1}b_{r-2}(\mathcal{H})+(-1)^{r}b_{r-1}(\mathcal{H})=0.$$
Moreover, since the simplicial complex $X(\mathcal{H})$ has a complete $(r-2)$-dimensional skeleton, we have that $b_{-1}(\mathcal{H})=b_0(\mathcal{H})=b_1(\mathcal{H})=\dots=b_{r-3}(\mathcal{H})=0$, and so we get that $$b_{r-2}(\mathcal{H})=b_{r-1}(\mathcal{H}).$$ This implies that a homogeneous  $d$-partition $(\mathcal{H}_1,\mathcal{H}_2,\dots, \mathcal{H}_d)$ of $K_{rd}^{(r)}$ is acyclic if and only if $b_{r-1}(\mathcal{H}_i)=0$ for all $1\leq i\leq d$. 
\end{remark}

\begin{remark} If $(\mathcal{H}_1,\mathcal{H}_2,\dots, \mathcal{H}_d)$  is a $d$-partition of  $K_{n}^{(r)}$ such that $X(\mathcal{H}_i)$ is  $\mathbb{Q}$-acyclic for each $1\leq i\leq d$ then $n=rd$. Indeed, the number of hyperedges in $K_n^{(r)}$ is ${\displaystyle \binom{n}{k}}$, while the number of hyperedges of the partition is ${\displaystyle d \binom{n-1}{r-1}}$. This implies that  
$$\frac{n!}{r!(n-r)!}=d\frac{(n-1)!}{(r-1)!(n-r)!},$$ which gives $n=rd$. When $r=2$ we recover the well-known fact that if the complete graph $K_n$ can be partitioned into $d$ spanning trees then $n=2d$. 
\end{remark}

\subsection{The map $det^{S^r}$}
The results in this subsection are needed for Corollary \ref{Comain}. 

The map $det^{S^r}$ is a byproduct of an exterior-algebra like construction $\Lambda_{V_d}^{S^r}$ from \cite{sta2,ls1}. It was conjectured in those papers that $dim_k(\Lambda_{V_d}^{S^r}[rd])=1$ for any $r$ and $d$. This conjecture is equivalent  with the existence and uniqueness of a nontrivial linear map 
$$det^{S^r}:V_d^{\otimes\binom{rd}{r}}\to k,$$ with the property that $det^{S^r}(\otimes_{1\leq i_1<i_2<\ldots<i_r\leq rd}(v_{i_1,\ldots,i_r}))=0$ if there exists $1\leq x_1<x_2<\ldots<x_{r+1}\leq rd$ such that \[v_{x_1,x_2,\ldots,x_r}=v_{x_1,\ldots,x_{r-1},x_{r+1}}=v_{x_1,\ldots,x_{r-2},x_r,x_{r+1}}=\ldots=v_{x_1,x_3,\ldots,x_r,x_{r+1}}=v_{x_2,x_3,\ldots,x_{r+1}}.\]

A linear map $det^{S^r}$ that satisfies the above condition was constructed in \cite{dets2} for $r=2$, and in \cite{ls2} for $r\geq 3$. It was proved that $det^{S^2}$ and $det^{S^3}$ are nontrivial,  while the case $r\geq 4$ remained  an open question at the time.

Let  $\{e_1,\dots,e_d\}$ be a basis for $V_d$, and take $(\mathcal{H}_1,\dots, \mathcal{H}_d)$ a hyperedge $d$-partition of the complete graph $K_{rd}^{(r)}$.   We define $$\omega_{(\mathcal{H}_1,\dots, \mathcal{H}_d)}=\otimes_{1\leq i_1\leq \dots \leq i_r\leq rd}(e_{i_1,\dots ,i_r})\in V_d^{\otimes\binom{rd}{r}},$$ determined by $e_{i_1,\dots ,i_r}=e_t$ if and only if $(i_1,\dots ,i_r)\in \mathcal{H}_t$. 

The following theorem gives the connection between hyperedge $d$-partitions of the complete graph $K_{rd}^{(r)}$ and the map $det^{S^r}$. The case $r=2$ was first discussed in \cite{dets2}. 
\begin{theorem}[\cite{ls2}] \label{ThCombR} Let $(\mathcal{H}_1,\mathcal{H}_2,\dots, \mathcal{H}_d)$ be a $d$-partition of $K_{rd}^{(r)}$. The following are equivalent.
\begin{enumerate} 
\item $det^{S^r}(\omega_{(\mathcal{H}_1,\mathcal{H}_2,\dots, \mathcal{H}_d)})\neq 0$,
\item  $(\mathcal{H}_1,\mathcal{H}_2,\dots, \mathcal{H}_d)$ is acyclic.
\end{enumerate}
\end{theorem}

The main result of this paper is proving the existence of an acyclic $d$-partition of the $r$-uniform complete hypergraph $K_{rd}^{(r)}$. In particular, this shows that the map $det^{S^r}$ is nontrivial for all $r$.

\section{A $d$-partition of the $r$-uniform complete hypergraph $K_{rd}^{(r)}$}

\label{section3}

\subsection{The partition $\mathcal{E}_d^{(r)}$}

In this subsection we construct $\mathcal{E}_d^{(r)}=(\Omega_1^{(r,d)}, \Omega_2^{(r,d)},\dots, \Omega_d^{(r,d)})$ a homogeneous hyperedge $d$-partition of the $r$-uniform complete hypergraph $K_{rd}^{(r)}$. This construction is inspired by results from the Appendix to  \cite{ls2} written in collaboration  with Anthony Passero.

Let $r>0$ and $d>0$ be two integers. For each $1\leq a\leq d$ we take 
$${\bf S}_a = \{ra - (r-1),ra-(r-2), \dots, ra\}.$$
One can see that the sets ${\bf S}_a$ for $1\leq a\leq d$ form a partition of the set $\{1,2,\dots,rd\}$. 

\begin{definition}
For each $1\leq a\leq d$ we define an $r$-uniform hypergraph $\Omega_a^{(r,d)}$  determined as follows 
$$V(\Omega_a^{(r,d)})=\{1,2,\dots,rd\},$$ 
and 
 $$E\left( \Omega_a^{(r,d)}\right) =\left\lbrace (i_1,\dots,i_r) \; \middle| \; 
 \begin{aligned}
  & 1 \leq i_1 <\dots < i_r \leq rd,\\
  & i_1 + \dots + i_r = t-1\; (mod \;r), \; \\
  & \text{and } i_t \in S_a
\end{aligned}
\right\rbrace.$$
\end{definition}
\begin{remark} Notice  that $\Omega_a^{(r,d)}$ is a sub-hypergraph of the  $r$-uniform complete hypergraph $K_{rd}^{(r)}$ for each $1\leq a\leq d$. Moreover, one can easily see that if $a\neq b$ then $$E(\Omega_a^{(r,d)})\cap E(\Omega_b^{(r,d)})=\emptyset. $$ Finally, one has that 
$$\bigcup_{a=1}^d E(\Omega_a^{(r,d)})=E(K_{rd}^{(r)}).$$ 
This means that $$\mathcal{E}_d^{(r)}=\left(\Omega_1^{(r,d)}, \Omega_2^{(r,d)},\dots, \Omega_d^{(r,d)}\right)$$ is a hyperedge $d$-partition of the  $r$-uniform complete hypergraph $K_{rd}^{(r)}$. 
\end{remark}

\begin{lemma} For each $1\leq a<d$  the hypergraphs $\Omega_a^{(r,d)}$ and $\Omega_{a+1}^{(r,d)}$ are isomorphic. \label{lemma2}
\end{lemma}
\begin{proof}  Fix $1\leq a < d$. It is enough to construct a bijective map $\phi:\{1,2,\dots,rd\}\to \{1,2,\dots,rd\},$ such that the induced map on $E(K_{rd}^{(r)})$ satisfies  
$$\phi(E(\Omega_{a+1}^{(r,d)}))\subseteq E(\Omega_{a}^{(r,d)}),\text{ and }\phi^{-1}(E(\Omega_{a}^{(r,d)}))\subseteq E(\Omega_{a+1}^{(r,d)}).$$

As noticed above 
$$\{1,2,\dots,rd\}=\bigcup_{k=1}^d {\bf S}_k,$$ is a partition of the set $\{1,2,\dots,rd\}$, and so it is enough to define the map $\phi$ on the sets ${\bf S}_k$ for all $1\leq k\leq d$. We take  $$\phi:\{1,2,\dots,rd\}\to \{1,2,\dots,rd\}$$ determined by 
\begin{itemize}
\item if $z\notin {\bf S}_a$ and $z\notin {\bf S}_{a+1}$ then $\phi(z)=z$,
\item if $z\in {\bf S}_{a+1}$ then $\phi(z)\in {\bf S}_a$ and $\phi(z)=z\; (mod\; r)$, 
\item if $z\in {\bf S}_{a}$ then $\phi(z)\in {\bf S}_{a+1}$ and $\phi(z)=z-1\; (mod\; r)$. 
\end{itemize}
One can check that the above conditions determine a well defined bijective map. 

Take $1\leq i_1<i_2<\dots<i_r\leq rd$ such that $(i_1,i_2,\dots,i_r)\in E(\Omega_{a+1}^{(r,d)})$. This means that there exists $1\leq t\leq r$ with  $$\sum_{l=1}^ri_l=(t-1)\;(mod\; r),\text{ and }i_t\in {\bf S}_{a+1}.$$

Let 
$$\phi((i_1,i_2,\dots,i_r))=(j_1,j_2,\dots,j_r),$$
where $1\leq j_1<j_2<\dots<j_r\leq rd$, and $\{\phi(i_1),\phi(i_2),\dots,\phi(i_r)\}=\{j_1,j_2,\dots,j_r\}$ as sets. 

For each $1\leq k\leq d$ we denote 
$$\alpha_k=|\{i_1,i_2,\dots,i_r\}\bigcap {\bf S}_k|,$$
$$\beta_k=|\{j_1,j_2,\dots,j_r\}\bigcap {\bf S}_k|.$$
One can show that 
 \begin{eqnarray}\label{eqj0}
\beta_k=\left\{
 \begin{aligned}
  & \alpha_{a+1} \text{ if } k=a,\\
  &  \alpha_a \text{ if } k=a+1,\\
  & \alpha_k \text{ otherwise}.
\end{aligned}
\right. 
\end{eqnarray}
Indeed, from the definition of $\phi$ we have 
\begin{eqnarray*}
\beta_{a}&=&|\{j_1,j_2,\dots,j_r\}\bigcap {\bf S}_a|\\
&=&|\{\phi(i_1),\phi(i_2),\dots,\phi(i_r)\}\bigcap {\bf S}_{a}|\\
&=&|\{i_1,i_2,\dots,i_r\}\bigcap {\bf S}_{a+1}|\\
&=&\alpha_{a+1}.
\end{eqnarray*}
The other identities can be proved in a similar fashion. 

Since $i_t\in {\bf S}_{a+1}$ we must have 
$$\alpha_1+\dots+\alpha_{a-1}+\alpha_a<t\leq \alpha_1+\dots+\alpha_{a-1}+\alpha_a+\alpha_{a+1}.$$ 
Subtracting $\alpha_a$ and using  Equation (\ref{eqj0}) we get
\begin{eqnarray*}
\sum_{k=1}^{a-1}\beta_k=\sum_{k=1}^{a-1}\alpha_k<t-\alpha_a\leq \left(\sum_{k=1}^{a-1}\alpha_k\right)+\alpha_{a+1}= \sum_{k=1}^{a}\beta_k,
\end{eqnarray*}
which means that  
\begin{eqnarray}j_{(t-\alpha_a)}\in {\bf S}_{a}. \label{eqj1}
\end{eqnarray} 
Next, since $\phi(z)=z-1\; (mod\; r)$ for each  $z\in {\bf S}_{a}$, $\phi(z)=z\; (mod\; r)$ if $z\notin {\bf S}_{a}$, and $\alpha_a=|\{i_1,i_2,\dots,i_r\}\bigcap {\bf S}_a|$, we have 
\begin{eqnarray}  \label{eqj2}
\begin{split}
j_1+j_2+\dots+j_r\; (mod\;r)&= \phi(i_1)+\phi(i_2)+\dots+\phi(i_r)\; (mod\;r)\\
&= (i_1+i_2+\dots+i_r)-\alpha_a\; (mod\;r)\\
&= (t-\alpha_a)-1\; (mod\;r).
 \end{split}
\end{eqnarray}
From Equations (\ref{eqj1}) and (\ref{eqj2})  we get that  $$\phi((i_1,i_2,\dots,i_r))=(j_1,j_2,\dots,j_r)\in E(\Omega_{a}^{(r,d)}),$$ which means that $\phi(E(\Omega_{a+1}^{(r,d)}))\subseteq E(\Omega_{a}^{(r,d)})$. 

A similar argument shows that $\phi^{-1}(E(\Omega_{a}^{(r,d)}))\subseteq E(\Omega_{a+1}^{(r,d)})$. 

\end{proof}
\begin{lemma} The $d$-partition $\mathcal{E}_d^{(r)}$ is homogeneous. \label{lemma3}
\end{lemma} 

\begin{proof}
First let's see that $E_{r-1}(\Omega_{1}^{(r,d)})=E(K_{rd}^{(r-1)})$. Take $(i_2,\dots,i_r)\in E(K_{rd}^{(r-1)})$, with $1\leq i_2<i_3<\dots <i_r\leq rd$. We will show that there exists $x\in {\bf S}_1$ such that $i_2,i_3, \dots ,i_r$ together with $x$ form a hyperedge  in $\Omega_{1}^{(r,d)}$. 

Let $k$ be the unique integer such that  $1\leq k\leq r$, $i_2,i_3,\dots,i_k\in {\bf S}_1$, and $i_{k+1},\dots,i_r\notin {\bf S}_1$. In order for $(i_2,i_3,\dots,i_r)$ to be in $E_{r-1}(\Omega_{1}^{(r,d)})$ it is enough to find $1\leq t\leq k$, and $x_t\in {\bf S}_1\setminus \{i_2, \dots, i_k\}$ such that   
\begin{eqnarray}
x_t+i_2+i_3+\dots+i_r&=&t-1\; (mod \;r). \label{eqforx}
\end{eqnarray}
Notice that Equation (\ref{eqforx}) depends on the parameter $t\in\{1,2,\dots,k\}$, and for each $t$ we have a unique solution $x_t\in {\bf S}_1=\{1,2,\dots, r\}$ determined by the condition 
$$x_t=t-1-i_2-i_3-\dots-i_r\; (mod \;r).$$   
Moreover,  one can see that $x_t\neq x_s$ for each $1\leq s\neq t\leq k$. This means that the set $\{x_1,x_2,\dots,x_{k}\}\subseteq {\bf S}_1$ has exactly $k$ elements. On the other hand the set $\{i_2,\dots,i_k\}\subseteq {\bf S}_1$ has exactly $k-1$ elements, and so there exist $1\leq t_0\leq k$ such that $x_{t_0}\notin \{i_2,\dots,i_k\}$.   This shows that  $$E_{r-1}(\Omega_{1}^{(r,d)})=E(K_{rd}^{(r-1)}).$$ 
Since each $\Omega_{i}^{(r,d)}$ is isomorphic to $\Omega_{1}^{(r,d)}$, we have that $E_{r-1}(\Omega_{i}^{(r,d)})=K_{rd}^{(r-1)}$, and so the $d$-partition $\mathcal{E}_d^{(r)}$ is pre-homogeneous. 

Finally, since  $\Omega_{i}^{(r,d)}$ is isomorphic to $\Omega_{1}^{(r,d)}$, the $d$-partition $\mathcal{E}_d^{(r)}$ divides evenly the  ${\displaystyle \binom{rd}{r}}$ hyperedges of $K_{rd}^{(r)}$  among the $d$ sub-hypergraphs $\Omega_{i}^{(r,d)}$. Hence,  
$$\vert E(\Omega_{i}^{(r,d)})\vert=\frac{1}{d}\vert E(K_{rd}^{(r)})\vert=\frac{1}{d}\binom{rd}{r},$$
which shows that $\mathcal{E}_d^{(r)}$ is homogeneous. 
\end{proof}

\subsection{The hypergraph $\Omega_{1}^{(r,d)}$} Next, we focus on the first component, the hypergraph $\Omega_{1}^{(r,d)}$. 

\begin{definition} Let $1\leq k\leq r$ and take $r+1\leq j_{k+1}<j_{k+2}<\dots <j_r\leq rd$. We define the $r$-uniform hypergraph $\Gamma_{k,r}(j_{k+1},\dots,j_{r})$ as the sub-hypergraph of $\Omega_{1}^{(r,d)}$  determined by 
$$V(\Gamma_{k,r}(j_{k+1},\dots,j_{r}))=\{1,2,\dots,rd\},$$

$$E(\Gamma_{k,r}\left(j_{k+1},\dots,j_{r})\right) =\left\lbrace (i_1,i_2,\dots,i_k,j_{k+1},\dots,j_r) \; \middle| \; 
 \begin{aligned}
  & 1\leq i_1<i_2<\dots <i_k\leq r,~~ \text{and}\\
  & (i_1,i_2,\dots,i_k,j_{k+1},\dots,j_r)\in E(\Omega_{1}^{(r,d)})
\end{aligned}
\right\rbrace.$$
\end{definition}
\begin{example} If $r=4$ and $d=3$, we have that  $E(\Gamma_{1,4}(5,7,12))=\{(4,5,7,12)\}$, and $E(\Gamma_{2,4}(5,7))=\{(1,3,5,7),(1,4,5,7),(2,3,5,7)\}$.  
\end{example}

\begin{remark}
Notice  that $\Omega_{1}^{(r,d)}$ is the disjoint union of the hypergraphs $\Gamma_{k,r}(j_{k+1},\dots,j_{r})$. More precisely, we have 
\begin{eqnarray*}
E(\Omega_{1}^{(r,d)})=\bigcup_{1\leq k\leq r}\left(\bigcup_{r+1\leq j_{k+1}<\dots<j_r\leq rd}E(\Gamma_{k,r}(j_{k+1},\dots,j_{r}))\right).
\end{eqnarray*}
\label{rem3}
Indeed, if $\sigma\in E(\Omega_{1}^{(r,d)})$ then there exist unique $1\leq k\leq r$ such that $\sigma=(i_1,\dots,i_k,j_{k+1},\dots,j_r)$, $1\leq i_1<\dots<i_k\leq r$, and $r+1\leq j_{k+1}<\dots<j_r\leq rd$, which means that $\sigma\in E(\Gamma_{k,r}(j_{k+1},\dots,j_{r}))$. 
Moreover, $(j_{k+1},\dots,j_{r})$ is unique with the above property.
\end{remark}

\begin{example}\label{exampleG23}
When $r=3$ and $d=2$ we have 
\begin{eqnarray*}
E(\Omega_{1}^{(3,2)})= \left(\bigcup_{4\leq j_2<j_3\leq 6}E(\Gamma_{1,3}(j_2,j_3))\right) \bigcup \left(\bigcup_{4\leq j_3\leq 6}E(\Gamma_{2,3}(j_3)) \right)  \bigcup E(\Gamma_{3,3}(\emptyset)).
\end{eqnarray*}
where  $$E(\Gamma_{1,3}(4,5))=\{(3,4,5)\}, ~~~ E(\Gamma_{1,3}(4,6))=\{(2,4,6)\}, ~~~E(\Gamma_{1,3}(5,6))=\{(1,5,6)\},$$ 
$$E(\Gamma_{2,3}(4))=\{(1,2,4),(2,3,4)\},~~~E(\Gamma_{2,3}(5))=\{(1,3,5),(2,3,5)\},~~~E(\Gamma_{2,3}(6))=\{(1,2,6),(1,3,6)\},$$ and 
$$E(\Gamma_{3,3}(\emptyset))=\{(1,2,3)\}.$$  
See  also the left side of Figure \ref{fig1}, where the hypergraphs $\Gamma_{1,3}(j_2,j_3)$ are in light blue, $\Gamma_{2,3}(j_3)$ in mid blue, and $\Gamma_{3,3}(\emptyset)$ in darker blue. 
\end{example}

\begin{definition} Let $1\leq k\leq r$ and $0\leq a\leq r-1$. We define the $k$-uniform hypergraph $\Gamma_{k,r}^a$  determined by 
$$V(\Gamma_{k,r}^a)=\{1,2,\dots,r\},$$ 

$$E(\Gamma_{k,r}^a)=\left\lbrace (i_1,i_2,\dots,i_k) \; \middle| \; 
 \begin{aligned}
  & 1\leq i_1<i_2<\dots i_k\leq r,~~ \text{and}\\
  & i_1+i_2+\dots+i_k=a+t ~(mod~r) ~\text{for some} ~0\leq t\leq k-1
\end{aligned}
\right\rbrace.$$
\end{definition}

\begin{example}
If $r=4$ and $d=3$, we have $E(\Gamma_{1,4}^{0})=\{(4)\}$, $E(\Gamma_{2,4}^{0})=\{(1,3),(1,4),(2,3)\}$, and  $E(\Gamma_{2,4}^{1})=\{(1,4),(2,4),(2,3)\}$.
\end{example}

\begin{remark}
One should notice that $X(\Gamma_{k,r}^a)$ is a particular case of the sum complex $X_A$ considered in \cite{lmr}. 
\end{remark}

\begin{remark} Let $1\leq k\leq r$. Consider $r+1\leq j_{k+1}<j_{k+2}<\dots <j_r\leq rd$, and take $a=j_{k+1}+\dots+j_r~(mod~r)$. 
Then the map $$\psi:E(\Gamma_{k,r}^{-a})\to E(\Gamma_{k,r}(j_{k+1},\dots,j_{r})),$$ determined by 
$$\psi(i_1,\dots,i_k)=(i_1,\dots,i_k,j_{k+1}\dots,j_r)$$ gives a bijection between the set of $k$-hyperedges of $\Gamma_{k,r}^{-a}$, and the set of $r$-hyperedges of $\Gamma_{k,r}(j_{k+1},\dots,j_{r})$. 
Indeed, $\psi$ is well defined because if $(i_1,\dots,i_k)\in E(\Gamma_{k,r}^{-a})$ then 
$$i_1+\dots+i_k=-a+t,$$ for some $0\leq t\leq k-1$. Since $j_{k+1}+\dots+j_r=a ~(mod~r)$ we have
$$i_1+\dots+i_k+j_{k+1}+\dots+j_r=t,$$
where $0\leq t\leq k-1$, which means that $(i_1,\dots,i_k,j_{k+1}\dots,j_r)\in E(\Gamma_{k,r}(j_{k+1},\dots,j_{r}))$. One can easily see that $\psi$ is bijective. However, $\psi$ is not an isomorphism of hypergraphs when $k\neq r$. \label{lemmagamma}
\end{remark}

\begin{remark} One can easily see that if $a=b~(mod~r)$, then $\Gamma_{k,r}^a=\Gamma_{k,r}^b$. Moreover, one can show that 
the map $\phi:\{1,\dots,r\}\to \{1,\dots,r\}$ given by  $\phi(i)=i+1\; (mod\;r)$ induces  an isomorphism $\Gamma_{k,r}^a\simeq \Gamma_{k,r}^{a+k}$. In general, $\Gamma_{k,r}^a$ and $\Gamma_{k,r}^b$ need not be isomorphic. For example,  $\Gamma_{3,6}^0$ is not isomorphic to $\Gamma_{3,6}^2$. Finally, one can prove  that $E_{k-1}(\Gamma_{k,r}^a)=K_{r}^{(k-1)}$,  and $\vert E(\Gamma_{k,r}^a)\vert={\displaystyle \binom{r-1}{k-1}}$. 
Since we do not need any of these results, we leave their proofs to the reader. 
\end{remark}

\section{Leaf-Equivalence} 
\label{section4}
In this section we introduce the leaf-equivalence relation among $k$-uniform hypergraphs, and show how it can be used to compute Betti numbers. 

\begin{definition} Let $\Gamma$ be a $k$-uniform hypergraph, and $\sigma=(i_1,\dots,i_k)\in E(\Gamma)$ such that there exits $\sigma_{\widehat{s}}=(i_1,\dots,\widehat{i_s},\dots,i_k)\in E_{k-1}(\Gamma)$, a $(k-1)$-face of $\sigma$ that is not a $(k-1)$-face for any other hyperedge of $\Gamma\setminus\{\sigma\}$. Then we say that $\Gamma$ and $\Gamma \setminus\{\sigma\}$ are leaf-equivalent via the $(k-1)$-face $\sigma_{\widehat{s}}$. 
\end{definition} 


\begin{definition}
Let $\Gamma=(E,V)$ be a $k$-uniform hypergraph. We say that $\Gamma$ is leaf-equivalent to the empty $k$-uniform hypergraph if there exist a sequence of $k$-uniform hypergraphs on the set $V$
$$\emptyset = \Gamma_0\subset \Gamma_1\subset \dots \subset \Gamma_n = \Gamma,$$
such that for each $1\leq i\leq n$, there exists a hyperedge $\sigma_i\in E(\Gamma_i)$ for which
$\Gamma_i$ and $\Gamma_{i-1}=\Gamma_i \setminus\{\sigma_i\}$ are leaf-equivalent via some $(k-1)$-face of $\sigma_i$.
\end{definition}

\begin{example} If $G$ is a graph then $G$ is  leaf-equivalent to the empty graph if and only if $G$ is a disjoint union of trees (i.e. $G$ is a forest). The $3$-uniform hypergraph on $4$ vertices $K_4^{(3)}$ is not leaf-equivalent to the empty $3$-uniform hypergraph on $4$ vertices since every $(3-1)$-face of $K_4^{(3)}$ is shared by two hyperedges. 

In Figure \ref{fig2} we present a step-by-step leaf-equivalence between a $3$-uniform hypergraph with $6$ vertices and $3$ hyperedges, and the $3$-uniform empty hypergraph with $6$ vertices. At each step, the $3-1$-face that is used for the leaf-equivalence is marked with $\|$. 

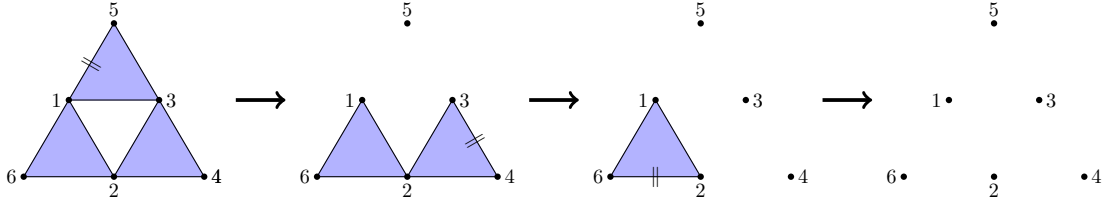
\begin{figure}[h!]
 \centering
 \begin{tikzpicture}[scale=0.6,every node/.style={scale=0.7}]
 %
 \draw[fill, color=blue!30]
(1,1.7) -- (2,0) -- (0,0) -- cycle;
\draw[fill, color=blue!30]
(4,0) -- (2,0) -- (3,1.7) -- cycle;
\draw[fill, color=blue!30]
(1,1.7) -- (2,3.4) -- (3,1.7) -- cycle;

\draw (1,1.7) node[left] {} -- (2,0)
node[below] {} -- (0,0)
node[left] {} -- cycle;


\draw (1,1.7) node[below] {} -- (2,3.4)
node[above] {} -- (3,1.7)
node[right] {} -- cycle;

\draw (4,0) node[right] {4} -- (2,0)
node[below] {} -- (3,1.7)
node[right] {} -- cycle;

\draw (1.5,2.5) node[sloped] {\rotatebox[origin=c]{60}{$\|$}};


 \fill (1,1.7) circle (2pt) node[left] {1};
 \fill (2,0) circle (2pt) node[below] {2};
 \fill (3,1.7) circle (2pt) node[right] {3};
 \fill (4,0) circle (2pt) node[right] {4};
 \fill (2,3.4) circle (2pt) node[above] {5};
 \fill (0,0) circle (2pt) node[left] {6};

\draw[->, line width=0.5mm] (4.7,1.7) -- (5.8,1.7);

 \draw[fill, color=blue!30]
(7.5,1.7) -- (8.5,0) -- (6.5,0) -- cycle;
\draw[fill, color=blue!30]
(10.5,0) -- (8.5,0) -- (9.5,1.7) -- cycle;

\draw (7.5,1.7) node[left] {} -- (8.5,0)
node[below] {} -- (6.5,0)
node[left] {} -- cycle;

\draw (10.5,0) node[right] {} -- (8.5,0)
node[below] {} -- (9.5,1.7)
node[right] {} -- cycle;

\draw (10,0.8) node[sloped] {\rotatebox[origin=c]{-60}{$\|$}};

 \fill (7.5,1.7) circle (2pt) node[left] {1};
 \fill (8.5,0) circle (2pt) node[below] {2};
 \fill (9.5,1.7) circle (2pt) node[right] {3};
 \fill (10.5,0) circle (2pt) node[right] {4};
 \fill (8.5,3.4) circle (2pt) node[above] {5};
 \fill (6.5,0) circle (2pt) node[left] {6};

\draw[->, line width=0.5mm] (11.2,1.7) -- (12.3,1.7);

 \draw[fill, color=blue!30]
(14,1.7) -- (15,0) -- (13,0) -- cycle;

\draw (13,0) node[right] {} -- (15,0)
node[below] {} -- (14,1.7)
node[right] {} -- cycle;

\draw (14,0) node[sloped] {$\|$};

 \fill (14,1.7) circle (2pt) node[left] {1};
 \fill (15,0) circle (2pt) node[below] {2};
 \fill (16,1.7) circle (2pt) node[right] {3};
 \fill (17,0) circle (2pt) node[right] {4};
 \fill (15,3.4) circle (2pt) node[above] {5};
 \fill (13,0) circle (2pt) node[left] {6};

\draw[->, line width=0.5mm] (17.7,1.7) -- (18.8,1.7);

 \fill (20.5,1.7) circle (2pt) node[left] {1};
 \fill (21.5,0) circle (2pt) node[below] {2};
 \fill (22.5,1.7) circle (2pt) node[right] {3};
 \fill (23.5,0) circle (2pt) node[right] {4};
 \fill (21.5,3.4) circle (2pt) node[above] {5};
 \fill (19.5,0) circle (2pt) node[left] {6};

 \end{tikzpicture}
 \caption{Leaf-equivalence to the empty $3$-uniform hypergraph \label{fig2}}
\end{figure}

\end{example}

Next, we introduce additional  notation. 
\begin{definition} Let $r$ be a positive integer, $1\leq k\leq r-1$, and $0\leq a\leq r-1$. 
For $\sigma=(i_1,\dots,i_k)\in E(\Gamma_{k,r}^a)$ we define the support of $\sigma$ as $$supp(\sigma)=\{i_1,\dots,i_k\},$$
and the weight of $\sigma$ as  $$w(\sigma)=\sum_{j=1}^ki_j.$$ 
If $G$ is a sub-hypergraph of $\Gamma_{k,r}^a$ we define the weight of $G$ as
$$q_a(G)=max\left\lbrace\left\lfloor \frac{w(\sigma)-a}{r} \right\rfloor ~\middle| ~\sigma \in E(G)\right\rbrace,$$
where $\lfloor x \rfloor$ is the floor of $x$.
Finally, we define the weight class of $G$ as 
$$Q_a(G)=\left\lbrace\sigma \in E(G)~\middle|~ \left\lfloor \frac{w(\sigma)-a}{r} \right\rfloor=q_a(G)\right\rbrace=\left\lbrace\sigma\in E(G)~|~w(\sigma)>a+(q_a(G)-1)r+k-1\right\rbrace.$$
\end{definition}
\begin{remark} Technically speaking, $supp(\sigma)$ is nothing else but the hyperedge $\sigma$. In Lemma \ref{lemma10} we need to use both notations (ordered and unordered presentation of hyperedges), so we introduce $supp(\sigma)$ to distinguish between the two settings. 

\end{remark}

\begin{example}  If $G=\Gamma_{3,5}^0$ then $supp((1,3,4))=\{1,3,4\}$, and $w((1,3,4))=1+3+4=8$. Moreover, one can see that $q_0(\Gamma_{3,5}^0)=2$, which is attained for example when $\sigma=(3,4,5)$ since $w((3,4,5))=12=0+2\cdot 5+2$.  Finally, we have 
$$Q_0(\Gamma_{3,5}^0)=\{(1,4,5),(2,3,5),(2,4,5),(3,4,5)\}.$$
\end{example}

\begin{lemma} \label{lemma10} Let $G$ be a sub-hypergraph of $\Gamma_{k,r}^a$. Then there exists $\sigma\in E(G)$ such that $G$ is leaf-equivalent to $G\setminus \{\sigma\}$. In particular, $\Gamma_{k,r}^a$ is leaf-equivalent to the empty $k$-uniform hypergraph. 
\end{lemma}
\begin{proof} We will prove this result by contradiction. Assume that there exists $G$ a sub-hypergraph of $\Gamma_{k,r}^a$ such that $G$ is not leaf-equivalent to $G\setminus \{\sigma\}$ for any $\sigma\in E(G)$. This means that for every $\sigma=(i_1,\dots,i_k)\in E(G)$, and any $1\leq s\leq k$ there exists a $\tau_s\in E(G)$, $\tau_s\neq \sigma$, such that $\sigma_{\widehat{s}}=(i_1,\dots,\hat{i_s},\dots,i_k)$ is a $(k-1)$-face of $\tau_s$. 

First, we construct an element $\theta_0=(m_0(1),\dots,m_0(k))\in Q_a(G)$ determined by 
\[
m_0(k) = 
 min\left\lbrace 
 x\in \{1,2,\dots, r\}\middle|\;
\begin{aligned}
& \text{there exist }i_1,i_2,\dots,i_{k-1}\in \{1,2,\dots,r\} \text{ such that}\\
 & 1\leq i_1<i_2<\dots<i_{k-1}<x\leq r, \text{ and }\\
 & ~(i_1,\dots,i_{k-1},x)\in Q_a(G)
\end{aligned}
\right\rbrace,
\]
and for $1\leq j<k$ we define recursively  
\[
m_0(j) = 
 min\left\lbrace 
 x\in \{1,2,\dots, r\}\middle|\;
\begin{aligned}
& \text{there exist }i_1,i_2,\dots,i_{j-1}\in \{1,2,\dots,r\} \text{ such that}\\
 & 1\leq i_1<\dots<i_{j-1}<x<m_0(j+1)<\dots<m_0(k)\leq r,  \text{ and }\\
 & ~(i_1,\dots,i_{j-1},x,m_0(j+1),\dots,m_0(k))\in Q_a(G)
\end{aligned}
\right\rbrace.
\]

It is obvious from our construction that $$\theta_0=(m_0(1),\dots,m_0(k))\in Q_a(G).$$ In particular 
$$w(\theta_0)=m_0(1)+\dots+m_0(k)=a+q_a(G)r+t_0$$
for some $0\leq t_0\leq k-1$.

The plan is to show that under our assumption (i.e. $G$ is not leaf-equivalent to $G\setminus \{\sigma\}$ for any $\sigma\in E(G)$), there exist $$0=j_0<j_1<\dots<j_{s-1}<j_s=k,$$ and $$\theta_u=(m_{u}(1),m_{u}(2),\dots,m_{u}(k))\in Q_a(G)$$ for each $0\leq u\leq s$ such that 
\begin{enumerate}
\item if $u<s$ then $supp(\theta_{u})\setminus \{m_u(j_u+1)\}=supp(\theta_{u+1})\setminus \{m_{u+1}(j_{u+1})\}$, 
\item $m_u(z)=m_0(z)$ for all $z\geq j_{u}+1$, 
\item $w(\theta_{u})\geq w(\theta_0)+j_u$. 
\end{enumerate}

For this we use induction. First we construct $j_1$ and $\theta_1$. Since $\theta_0\in Q_a(G)\subseteq E(G)$ it means that the $(k-1)$-face $(m_0(2),\dots,m_0(k))$ must be shared by some other hyperedge $\theta_1=(m_1(1),\dots,m_1(k))\in E(G)$, with $\theta_1\neq \theta_0$. In other words, there exists $x_1\notin supp(\theta_0)$ such that 
$$supp(\theta_{1})\setminus \{x_1\}=supp(\theta_{0})\setminus \{m_0(1)\}=\{m_0(2),\dots,m_0(k)\}.$$
With this notation we have
$$w(\theta_1)=w(\theta_0)+x_1-m_0(1).$$ 
Notice that $x_1\geq 1$, and $m_0(1)\leq r-k+1$ (that is because $1\leq m_0(1)<m_0(2)<\dots<m_0(k)\leq r$). This implies that $x_1-m_0(1)\geq k-r$, thus 
\begin{eqnarray*}w(\theta_1)&=&w(\theta_0)+x_1-m_0(1)\\
&\geq& a+q_a(G)r+t_0+k-r\\
&=&a+(q_a(G)-1)r+k+t_0\\
&>&a+(q_a(G)-1)r+k-1,
\end{eqnarray*} which means that $\theta_1\in Q_a(G)$. 

Since $m_0(2),m_0(3),\dots,m_0(k)\in supp(\theta_0)\cap  supp(\theta_1)$, by the minimality of $m_0(1)$ it follows that $m_0(1)<x_1$, and so there exists $j_0=0<j_1\leq k$ such that $$m_0(j_1)<x_1<m_0(j_1+1).$$

To summarize, we have  
 \begin{eqnarray*}
m_1(z)=\left\{
 \begin{aligned}
&m_0(z+1)  \text{ for }1\leq z<j_1,\\ 
&x_{1} \text{ for } z=j_{1}\\
&m_0(z) \text{ for }j_1+1\leq z\leq k. 
\end{aligned}
\right. 
\end{eqnarray*}

Finally, because $1\leq m_0(1)<m_0(2)<\dots <m_0(j_1)<x_1=m_1(j_1)<m_1(j_1+1)=m_0(j_1+1)$, we get that $x_1-m_0(1)=m_1(j_1)-m_0(1)\geq j_1$. Therefore,  
$$w(\theta_1)=w(\theta_0)+m_1(j_1)-m_0(1)\geq w(\theta_0)+j_1,$$
which concludes the construction of $j_1$ and $\theta_1$ with the desired properties.

Next, assume that we constructed  
$$0=j_0<j_1<\dots<j_{u}<k,$$ and $$\theta_t=(m_{t}(1),m_{t}(2),\dots,m_{t}(k))\in Q_a(G)$$ for each $0\leq t\leq u$ with the desired properties. We want to find $j_u<j_{u+1}\leq k$, and $\theta_{u+1}=(m_{u+1}(1),m_{u+1}(2),\dots,m_{u+1}(k))\in Q_a(G)$ that extend our construction.

Just like above, the $(k-1)$-face 
$(m_u(1),\dots,m_u(j_u),\widehat{m_u(j_u+1)},m_u(j_u+2),\dots,m_u(k))$ of  $\theta_u\in Q(G)\subseteq E(G)$  must be shared by some other hyperedge $\theta_{u+1}=(m_{u+1}(1),\dots,m_{u+1}(k))\in E(G)$ with $\theta_{u+1}\neq \theta_u$. In other words, there exists $x_{u+1}\notin supp(\theta_u)$ such that 
$$supp(\theta_{u+1})\setminus \{x_{u+1}\}=supp(\theta_{u})\setminus \{m_u(j_u+1)=m_0(j_u+1)\}=\{m_{u}(1),\dots,\widehat{m_u(j_u+1)},\dots,
m_u(k)\}.$$
In particular we have 
$$w(\theta_{u+1})=w(\theta_u)+x_{u+1}-m_u(j_u+1).$$ 
Notice that $x_{u+1}\geq 1$, and $m_u(j_u+1)\leq r-k+j_u+1$ (that is because $1\leq m_u(j_u+1)<m_u(j_u+2)<\dots<m_u(k)\leq r$). This implies that $x_{u+1}-m_u(j_u+1)\geq k-r-j_u$, and so 
\begin{eqnarray*}
w(\theta_{u+1})&=&w(\theta_u)+x_{u+1}-m_u(j_u+1)\\
&\geq& w(\theta_0)+j_u +k-r-j_u\\
&=&a+q_a(G)r+t_0+k-r\\
&=&a+(q_a(G)-1)r+k+t_0\\
&>&a+(q_a(G)-1)r+k-1,
\end{eqnarray*} which means that $\theta_{u+1}\in Q_a(G)$. 

Since $m_u(1),\dots,m_u(j_u),m_0(j_u+2)=m_u(j_u+2),m_0(j_u+3)=m_u(j_u+3),\dots,m_0(k)=m_u(k)\in supp(\theta_u)\cap supp(\theta_{u+1})$, by the minimality of $m_0(j_u+1)=m_u(j_u+1)$, it follows that $m_u(j_u+1)=m_0(j_u+1)<x_{u+1}$. Thus there exists $j_{u}<j_{u+1}\leq k$ such that $$m_u(j_{u+1})<x_{u+1}<m_u(j_{u+1}+1).$$ 

To summarize, we have  
 \begin{eqnarray*}
m_{u+1}(z)=\left\{
 \begin{aligned}
&m_u(z)  \text{ for }1\leq z\leq j_{u},\\
&m_u(z+1)  \text{ for }j_u+1\leq z<j_{u+1},\\ 
&x_{u+1} \text{ for } z=j_{u+1},\\
&m_u(z)=m_0(z) \text{ for }j_{u+1}+1\leq z\leq k.  
\end{aligned}
\right. 
\end{eqnarray*}


Moreover, because $1\leq m_u(j_u+1)<m_u(j_u+2)<\dots <m_u(j_{u+1})<x_{u+1}=m_{u+1}(j_{u+1})$ we get that $$x_{u+1}-m_u(j_{u}+1)=m_{u+1}(j_{u+1})-m_u(j_{u}+1)\geq j_{u+1}-j_u.$$ Therefore, 
$$w(\theta_{u+1})=w(\theta_u)+m_{u+1}(j_{u+1})-m_u(j_{u}+1)\geq w(\theta_0)+j_u+(j_{u+1}-j_u)=w(\theta_0)+j_{u+1}.$$

Since the sequence $j_0<j_1<..<j_u<...\leq k$ is strictly increasing, it means that there exists $s>0$ such that $j_s=k$, which completes our construction.

To finalize the proof, notice that the above construction  leads us to a contradiction. Indeed, 
\begin{eqnarray*}
w(\theta_s)&\geq& w(\theta_0)+j_s\\
&\geq& a+q(G)r+t_0+k\\
&>&a+q(G)r+k-1,
\end{eqnarray*} which means that  $$\left\lfloor \frac{w(\theta_s)-a}{r} \right\rfloor >q_a(G).$$ 
This obviously contradicts the definition of  $q_a(G)$. Hence, our assumption that $G$ is not leaf-equivalent to $G\setminus \{\sigma\}$ for any $\sigma\in E(G)$ is false. 
\end{proof}

\begin{remark}
In the particular case when $r$ is a prime number,  Lemma \ref{lemma10} is a consequence of the results from \cite{lmr}.  However, the case when $r$ is not a prime number is not treated there. 
\end{remark}

\begin{lemma} Let $1\leq k\leq r$ and $r+1\leq j_{k+1}<j_{k+2}<\dots <j_r\leq rd$. Then the  $r$-uniform hypergraph $\Gamma_{k,r}(j_{k+1},\dots,j_{r})$  is leaf-equivalent to the empty $r$-uniform hypergraph. Moreover, this leaf-equivalence is achieved via  $(r-1)$-faces of type  $(i_1,\dots, \hat{i_s},\dots,i_k,j_{k+1},\dots,j_r)$ (i.e. those that contain all the $j's$). \label{lemma4}
\end{lemma}
\begin{proof}  Take $a=j_{k+1}+\dots+j_r~(mod~r)$. As we noticed in Remark \ref{lemmagamma},  the map $\psi:E(\Gamma_{k,r}^{-a})\to E(\Gamma_{k,r}(j_{k+1},\dots,j_{r}))$ determined by 
$$\psi((i_1,\dots,i_k))=(i_1,\dots,i_k,j_{k+1},\dots,j_r)$$ 
gives a bijection between $E(\Gamma_{k,r}^{-a})$ and  $E(\Gamma_{k,r}(j_{k+1},\dots,j_{r}))$. 

Moreover, consider the map \( \xi: E_{k-1}( \Gamma_{k,r}^{-a}) \to E_{r-1}\left( \Gamma_{k,r}(j_{k+1}, \dots, j_r) \right) \) defined by
   \[ \xi( (i_1, \dots, \widehat{i_s}, \dots, i_k) ) = (i_1, \dots, \widehat{i_s}, \dots, i_k, j_{k+1}, \dots, j_r ).\]
Notice that while the map \( \xi \) is not surjective on the set \( E_{r-1}\left( \Gamma_{k,r}(j_{k+1}, \dots, j_r) \right) \), it induces a bijection onto its image. Therefore, we may view this induced map \( \xi \) as a bijection 
\begin{eqnarray*} \xi: E_{k-1}( \Gamma_{k,r}^{-a})   \xrightarrow{\sim} JE_{r-1}(\Gamma_{k,r}(j_{k+1},\dots,j_r))
\end{eqnarray*}
where  
\begin{eqnarray*}  JE_{r-1}(\Gamma_{k,r}(j_{k+1},\dots,j_r))=\left\lbrace \sigma_{\widehat{s}} \in E_{r-1}(\Gamma_{k,r}(j_{k+1},\dots,j_r)) \; \middle| \; \begin{aligned}
& \sigma_{\widehat{s}}=(i_1,\dots, \widehat{i_s},\dots,i_k,j_{k+1},\dots,j_r)\\
&\text{ where  } 1\leq i_1<\dots<i_k\leq r
\end{aligned} \right\rbrace, 
\end{eqnarray*}
is the set of those $(r-1)$-faces in $\Gamma_{k,r}(j_{k+1},\dots,j_r)$ that contain all the $j's$. 

Since $\Gamma_{k,r}^{-a}$ is leaf-equivalent to the empty $k$-uniform hypergraph, we can order its hyperedges $\sigma_1, \sigma_2,\dots,\sigma_n$ such that the $k$-uniform hypergraphs 
$$\Gamma_i=\Gamma_{k,r}^{-a}\setminus \{\sigma_{i+1},\sigma_{i+2},\dots,\sigma_n\}~\text{ for } 1\leq i\leq n,$$ give a leaf-equivalence between $\Gamma_{k,r}^{-a}$ and the empty $k$-uniform hypergraph. We denote by $f_i$ the $(k-1)$-face of $\sigma_i$ that does not belong to any another hyperedge of $\Gamma_i$. 

With this notation, one can see that the $r$-uniform hypergraphs  
$$\Delta_i=\Gamma_{k,r}(j_{k+1},\dots,j_{r})\setminus \{\psi(\sigma_{i+1}),\psi(\sigma_{i+2}),\dots,\psi(\sigma_n)\}~\text{ for } 1\leq i\leq n,$$ 
give a leaf-equivalence between $\Gamma_{k,r}(j_{k+1},\dots,j_{r})$, and the empty $r$-uniform hypergraph. Indeed, since $E(\Gamma_{k,r}^{-a})=\{\sigma_1,\dots,\sigma_n\}$ and $\psi$ is bijective, we get that 
$$E(\Gamma_{k,r}(j_{k+1},\dots,j_{r}))=\{\psi(\sigma_{1}),\psi(\sigma_{2}),\dots,\psi(\sigma_n)\}.$$ 
So, 
\begin{eqnarray*}
E(\Delta_{i-1})&=& E(\Gamma_{k,r}(j_{k+1},\dots,j_{r}))\setminus \{\psi(\sigma_{i}),\psi(\sigma_{i+1}),\dots,\psi(\sigma_n)\}\\
&=& \{\psi(\sigma_{1}),\psi(\sigma_{2}),\dots,\psi(\sigma_{i-1})\}\\
&=&E(\Delta_i) \setminus \{\psi(\sigma_{i})\}.
\end{eqnarray*}
Finally, notice that $\xi(f_i)\in JE_{r-1}(\Gamma_{k,r}(j_{k+1},\dots,j_r))$ is an $(r-1)$-face of $\psi(\sigma_{i})$  that does not belong to any other hyperedge of $\Delta_i$. 
Thus, $\Delta_{i-1}$ and $\Delta_i$ are leaf-equivalent via the $(r-1)$-face \(\xi(f_i)\), showing that \(\Gamma_{k,r}(j_{k+1}, \dots , j_r)\) is leaf-equivalent to the empty \(r\)-uniform hypergraph.  
\end{proof}
The following result is obvious, but for completeness we give a proof. 

\begin{lemma} Let $\Gamma$ be a $k$-uniform hypergraph and $\sigma=(i_1,\dots,i_k)\in E(\Gamma)=E_k(\Gamma)$ such that $\Gamma$ and 
$\Gamma \setminus\{\sigma\}$ are leaf-equivalent.  Then $b_{k-1}(\Gamma)=b_{k-1}(\Gamma \setminus\{\sigma\})$. In particular, if a $k$-uniform hypergraph $\Gamma$ is leaf-equivalent to the empty $k$-uniform hypergraph then $b_{k-1}(\Gamma)=0$. \label{lemmaC}
\end{lemma}
\begin{proof} 
Since  $\Gamma$ is $k$-uniform, the simplicial complex  $X(\Gamma)$ does not have cells of dimension strictly larger than $k-1$. Thus, $H_{k-1}(X(\Gamma))=Z_{k-1}(X(\Gamma))$. 

Let $\sigma_{\widehat{s}}=(i_1,\dots,\widehat{i_s},\dots,i_k)$ be the $(k-1)$-face of the hyperedge $\sigma$ that does not belong to any other hyperedge of $\Gamma$. We denote  $[\sigma]=[i_1,i_2,\dots,i_k]\in C_{k-1}(X(\Gamma))$, the $(k-1)$-dimensional cell of $X(\Gamma)$ corresponding to $\sigma$, and $[\sigma_{\widehat{s}}]=[i_1,\dots,\hat{i_s},\dots,i_k]\in C_{k-2}(X(\Gamma))$, the $(k-2)$-dimensional cell of $X(\Gamma)$ corresponding to the face $\sigma_{\widehat{s}}$. 

Let 
$\partial_{k-1}:C_{k-1}(X(\Gamma))\to C_{k-2}(X(\Gamma))$ be the boundary map for the homology complex. Take $$\sum_{i=1}^n\alpha_i[\tau_i]+\alpha[\sigma]\in Ker(\partial_{k-1}),$$ where $\tau_i$ are hyperedges in $\Gamma \setminus\{\sigma\}$, and $\alpha_i, \alpha\in \mathbb{Q}$. We have
\begin{eqnarray*}
0&=&\partial_{k-1}(\sum_{i=1}^n\alpha_i[\tau_i]+\alpha[\sigma])\\
&=&\sum_{i=1}^n\alpha_i\partial_{k-1}([\tau_i])+\alpha\partial_{k-1}([\sigma])\\
&=& \sum_{i=1}^n\alpha_i\partial_{k-1}([\tau_i])+\alpha(\sum_{j=1}^k(-1)^{j-1}[i_1,\dots,\hat{i_j},\dots,i_k]). 
\end{eqnarray*}
Since the $(k-1)$-face $\sigma_{\widehat{s}}=(i_1,\dots,\hat{i_s},\dots,i_k)$ does not belong to any hyperedge of $\Gamma\setminus\{\sigma\}$, it follows that on the right side of the above equality the coefficient of $[\sigma_{\widehat{s}}]=[i_1,\dots,\hat{i_s},\dots,i_k]$ is $(-1)^{s-1}\alpha$. On the other hand, on the left side of the equality the coefficient of $[\sigma_{\widehat{s}}]$ is $0$. Therefore, we must have $\alpha=0$. In other words,  we get that  $Z_{k-1}(X(\Gamma))=Z_{k-1}(X(\Gamma \setminus\{\sigma\}))$, which implies that $b_{k-1}(\Gamma)=b_{k-1}(\Gamma \setminus\{\sigma\})$. 
\end{proof}

\begin{remark} In general, if we remove a hyperedge $\sigma$ from a $k$-uniform hypergraph $\Gamma$, it is not true that $\Gamma$ and $\Gamma \setminus\{\sigma\}$ will have the same Betti numbers for all $i\leq k-1$ (even if the condition from Lemma \ref{lemmaC} is satisfied). For example, consider the $3$-uniform hypergraph $\Gamma$ from Figure \ref{fig2} determined by  $V(\Gamma)=\{1,2,3,4,5,6\}$, $E(\Gamma)=\{(1,2,6),(1,3,5),(2,3,4)\}$, and take $\sigma=(1,3,5)$. One can see that $b_{2}(\Gamma)=b_{2}(\Gamma \setminus\{\sigma\})$ (this follows from Lemma \ref{lemmaC}),  but $b_{1}(\Gamma)=1$ while $b_{1}(\Gamma \setminus\{\sigma\})=0$. 
\label{rembetti}
\end{remark}

\section{The Main Result}
\label{section5}
We are now ready to prove the main result of this paper. 

\begin{theorem}  \label{Th17r}  Let $r\geq 1$ and $d\geq 1$ two integers. Then $\mathcal{E}_d^{(r)}=(\Omega_1^{(r,d)}, \Omega_2^{(r,d)},\dots, \Omega_d^{(r,d)})$ is an acyclic  $d$-partition of the $r$-uniform  complete hypergraph $K_{rd}^{(r)}$. 
\end{theorem} 
\begin{proof} We already know that $\mathcal{E}_d^{(r)}$ is homogeneous. Since the hypergraphs $\Omega_i^{(r,d)}$ and $\Omega_1^{(r,d)}$ are isomorphic, from Remark \ref{rem211} it is enough to prove that $b_{r-1}(\Omega_{1}^{(r,d)})=0$. For this we will show that $\Omega_{1}^{(r,d)}$ is leaf-equivalent to the empty $r$-uniform hypergraph, and then use Lemma \ref{lemmaC}.

Recall from Remark \ref{rem3} that  
\begin{eqnarray*}
E(\Omega_{1}^{(r,d)})=\bigcup_{1\leq k\leq r}\left(\bigcup_{r+1\leq j_{k+1}<\dots<j_r\leq rd}E(\Gamma_{k,r}(j_{k+1},\dots,j_{r}))\right).
\end{eqnarray*}
 We will show that  $\Omega_{1}^{(r,d)}$ is leaf-equivalent to the empty $r$-uniform hypergraph by removing each $\Gamma_{k,r}(j_{k+1},\dots,j_{r})$ one at the time. We will use induction over $k$.

If $k=1$ then for each $r+1\leq j_2<j_3<\dots<j_r\leq rd$ there exists a unique   $1\leq x\leq r$ such that $x+j_2+\dots+j_r=0~(mod~r)$. Consequently, there is exactly one hyperedge $\sigma=(x,j_2,\dots,j_r)\in E(\Omega_{1}^{(r,d)})$, which means that $E(\Gamma_{1,r}(j_{2},\dots,j_{r}))=\{\sigma\}$. In particular, we get that the $(r-1)$-face $(j_2,\dots,j_r)\in JE_{r-1}(\Gamma_{1,r}(j_{2},\dots,j_r))=\{(j_2,\dots,j_r)\}$ is not shared by any other hyperedge in $E(\Omega_{1}^{(r,d)})$. This means that $\Omega_{1}^{(r,d)}$ and $\Omega_{1}^{(r,d)}\setminus\{\sigma\}$ are leaf-equivalent via an $(r-1)$-face in $JE_{r-1}(\Gamma_{1,r}(j_{2},\dots,j_r))$.

Also, notice that if $(j_2,\dots,j_r)\neq (l_2,\dots,l_r)$, then  
$$JE_{r-1}(\Gamma_{1,r}(j_{2},\dots,j_r))\cap E_{r-1}(\Gamma_{1,r}(l_{2},\dots,l_r))=\emptyset.$$
Therefore, the process of removing $\Gamma_{1,r}(j_{2},\dots,j_{r})$ and $\Gamma_{1,r}(l_{2},\dots,l_{r})$ can be done independently. This implies that  we can can use leaf-equivalence to remove every hyperedge in $$\bigcup_{r+1\leq j_2<\dots<j_r\leq rd}\Gamma_{1,r}(j_{2},\dots,j_{r}).$$ Hence, 
$$\Omega_{1}^{(r,d)} ~\text{ and } ~\Omega_{1}^{(r,d)}\setminus \bigcup_{r+1\leq j_2<\dots,j_r\leq rd}\Gamma_{1,r}(j_{2},\dots,j_{r})$$ are leaf-equivalent.

Next, assume that the induction hypothesis is true for $k$. This means that 
$$\Omega_{1}^{(r,d)} ~~~ \text{  and  } ~~~ \Omega_{1}^{(r,d)}\setminus \bigcup_{t=1}^k\left( \bigcup_{r+1\leq j_{t+1}<\dots<j_r\leq rd}\Gamma_{t,r}(j_{t+1},\dots,j_{r})\right)$$ are leaf-equivalent. We want to show that we can use leaf-equivalence to further remove each $\Gamma_{k+1,r}(j_{k+2},\dots,j_{r})$. 

We know from Lemma \ref{lemma4}  that for every $r+1\leq j_{k+2}<\dots<j_r\leq rd$, the $r$-uniform hypergraph $\Gamma_{k+1,r}(j_{k+2},\dots,j_{r})$ is leaf-equivalent to the empty $r$-uniform hypergraph via 
$(r-1)$-faces of the type $(i_1,\dots, \widehat{i_s},\dots,i_{k+1},j_{k+2},\dots,j_r)\in JE_{r-1}(\Gamma_{k+1,r}(j_{k+2},\dots,j_r))$ (i.e. those $(r-1)$-faces that contain all the $j_{k+2},\dots, j_r$). 

Moreover, notice that  
$$JE_{r-1}(\Gamma_{k+1,r}(j_{k+2},\dots,j_r))\cap E_{r-1}(\Gamma_{t,r}(l_{t+1},\dots,l_r))=\emptyset$$
for any $t> k+1$, or  $t=k+1$ and $(l_{k+2},\dots,l_k)\neq (j_{k+2},\dots,j_k)$. 

Indeed, if $\sigma_{\widehat{s}}\in JE_{r-1}(\Gamma_{k+1,r}(j_{k+2},\dots,j_r))$ then $\sigma_{\widehat{s}}=(i_1,\dots,\widehat{i_s},\dots,i_{k+1},j_{k+2},\dots,j_r)$ where $1\leq i_1<\dots<i_{k+1}\leq r$. If $t>k+1$ then $\sigma_{\widehat{s}}$ cannot be an element in $E_{r-1}(\Gamma_{t,r}(l_{t+1},\dots,l_r))$  since it does not  have enough vertices in the set $\{1,\dots,r\}$. Moreover, when $t=k+1$ one can see that if $\sigma_{\widehat{s}}\in E_{r-1}(\Gamma_{k+1,r}(l_{k+2},\dots,l_r))$ then $(l_{k+2},\dots,l_k)=(j_{k+2},\dots,j_k)$. 

Next, since 
$$\Omega_{1}^{(r,d)}\setminus \left(\bigcup_{t=1}^k\left(\bigcup_{r+1\leq j_{t+1}<\dots<j_r\leq rd}\Gamma_{t,r}(j_{t+1},\dots,j_{r})\right)\right)=\bigcup_{k+1\leq t\leq r}\left(\bigcup_{r+1\leq j_{t+1}<\dots<j_r\leq rd}\Gamma_{t,r}(j_{t+1},\dots,j_{r})\right),$$
 we get that the leaf-equivalence from $\Gamma_{k+1,r}(j_{k+2},\dots,j_{r})$  to the empty $r$-uniform hypergraph (via $(r-1)$-faces from $JE_{r-1}(\Gamma_{k+1,r}(j_{k+2},\dots,j_r))$), induces a leaf-equivalence from $$\Omega_{1}^{(r,d)}\setminus \left(\bigcup_{t=1}^k\left(\bigcup_{r+1\leq j_{t+1}<\dots<j_r\leq rd}\Gamma_{t,r}(j_{t+1},\dots,j_{r})\right)\right),$$ to 
$$\Omega_{1}^{(r,d)}\setminus \left(\bigcup_{t=1}^k\left(\bigcup_{r+1\leq j_{t+1}<\dots<j_r\leq rd}\Gamma_{t,r}(j_{t+1},\dots,j_{r})\right)\right)\setminus \Gamma_{k+1,r}(j_{k+2},\dots,j_{r}).$$

Finally notice that the leaf-equivalence at step $k+1$ is done independently for each $(j_{k+2},\dots,j_{r})$ (i.e. we do not use the same $(r-1)$-faces). Therefore  the induction hypothesis is true for $k+1$, which proves our statement. 

\end{proof}

\begin{corollary}\label{Comain} The map $det^{S^r}$ introduced in \cite{ls2} is nontrivial. In particular $dim_k(\Lambda_{V_d}^{S^r}[rd])\geq 1$. 
\end{corollary} 
\begin{proof} It follows from Theorem \ref{ThCombR} and Theorem \ref{Th17r}. For the second part see the discussion in \cite{ls2}. 
\end{proof}

\begin{remark} When $r=2$ the graph $\Omega_1^{(2,d)}$ is the twin-star graph $TS_d$ on $2d$ vertices  from Figure \ref{fig3}  (see also \cite{fls}). Notice that  $TS_d$ consists of a central edge $(1,2)$, with the other $2d-2$ edges being rays connected  either to vertex $1$, or to vertex $2$. 

\begin{figure}[h]
\centering
\begin{tikzpicture}
  [scale=0.6,auto=left]
	\node[shape=circle,draw=black,minimum size = 24pt,inner sep=0.3pt] (n1) at (3,0) {$1$};
	\node[shape=circle,draw=black,minimum size = 24pt,inner sep=0.3pt] (n2) at (11,0) {$2$};
	\node[shape=circle,draw=black,minimum size = 24pt,inner sep=0.3pt] (n3) at (8.35,1.3) {$4$};
  \node[shape=circle,draw=black,minimum size = 24pt,inner sep=0.3pt] (n4) at (5.65,1.3) {$3$};
	\node[shape=circle,draw=black,minimum size = 24pt,inner sep=0.3pt] (n5) at (9.9,2.5) {$6$};
  \node[shape=circle,draw=black,minimum size = 24pt,inner sep=0.3pt] (n6) at (4.1,2.5) {$5$};
	\node[shape=circle,draw=black,minimum size = 24pt,inner sep=0.3pt] (n51) at (12.1,2.5) {$8$};
  \node[shape=circle,draw=black,minimum size = 24pt,inner sep=0.3pt] (n61) at (1.9,2.5) {$7$};
	\node[shape=circle,draw=black,minimum size = 24pt,inner sep=0.3pt] (n71) at (12.1,-2.5) {{\tiny $2d-4$}};
  \node[shape=circle,draw=black,minimum size = 24pt,inner sep=0.3pt] (n81) at (1.9,-2.5) {{\tiny $2d-5$}};
	\node[shape=circle,draw=black,minimum size = 24pt,inner sep=0.3pt] (n7) at (9.9,-2.5) {{\tiny $2d-2$}};
  \node[shape=circle,draw=black,minimum size = 24pt,inner sep=0.3pt] (n8) at (4.1,-2.5) {{\tiny $2d-3$}};
	\node[shape=circle,draw=black,minimum size = 24pt,inner sep=0.3pt] (n9) at (8.35,-1.3) {{\tiny $2d$}};
  \node[shape=circle,draw=black,minimum size = 24pt,inner sep=0.3pt] (n10) at (5.65,-1.3) {{\tiny $2d-1$}};

	\node[shape=circle,minimum size = 24pt,inner sep=0.3pt] (m4) at (1.3,1) {{ \bf $\boldsymbol{\cdot}$}};
	\node[shape=circle,minimum size = 24pt,inner sep=0.3pt] (m4) at (1,0.3) {{ \bf $\boldsymbol{\cdot}$}};
	\node[shape=circle,minimum size = 24pt,inner sep=0.3pt] (m4) at (1,-0.3) {{ \bf $\boldsymbol{\cdot}$}};
	\node[shape=circle,minimum size = 24pt,inner sep=0.3pt] (m4) at (1.3,-1) {{ \bf $\boldsymbol{\cdot}$}};
	
	\node[shape=circle,minimum size = 24pt,inner sep=0.3pt] (m4) at (12.7,1) {{ \bf $\boldsymbol{\cdot}$}};
	\node[shape=circle,minimum size = 24pt,inner sep=0.3pt] (m4) at (13,0.3) {{ \bf $\boldsymbol{\cdot}$}};
	\node[shape=circle,minimum size = 24pt,inner sep=0.3pt] (m4) at (13,-0.3) {{ \bf $\boldsymbol{\cdot}$}};
	\node[shape=circle,minimum size = 24pt,inner sep=0.3pt] (m4) at (12.7,-1) {{ \bf $\boldsymbol{\cdot}$}};

	  \draw[line width=0.5mm,blue]  (n1) -- (n2)  ;
		\draw[line width=0.5mm,blue]  (n1) -- (n4)  ;
		\draw[line width=0.5mm,blue]  (n1) -- (n6)  ;
		\draw[line width=0.5mm,blue]  (n1) -- (n61)  ;
		\draw[line width=0.5mm,blue]  (n1) -- (n8)  ;
		\draw[line width=0.5mm,blue]  (n1) -- (n81)  ;
		\draw[line width=0.5mm,blue]  (n1) -- (n10)  ;
		\draw[line width=0.5mm,blue]  (n3) -- (n2)  ;
	  \draw[line width=0.5mm,blue]  (n5) -- (n2)  ;
		\draw[line width=0.5mm,blue]  (n51) -- (n2)  ;
	  \draw[line width=0.5mm,blue]  (n7) -- (n2)  ;
	  \draw[line width=0.5mm,blue]  (n71) -- (n2)  ;
		\draw[line width=0.5mm,blue]  (n9) -- (n2)  ;
\end{tikzpicture}
\caption{$TS_d$ the Twin-Star graph  with $2d$ vertices} \label{fig3}
\end{figure}
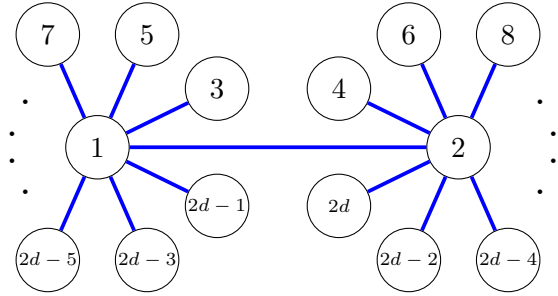

Similarly when \(r \geq 3\), the \(r\)-uniform hypergraph $\Omega_1^{(r,d)}$ consists of a central hyperedge $(1,\dots,r)$ along with the remaining hyperedges $\sigma\in \Omega_1^{(r,d)}$ having the property that
$$supp(\sigma)\cap \{1,2,\dots,r\} \neq \emptyset,$$
and so can be viewed as hyper-rays. Unlike for graphs, there are different flavors of hyper-rays depending on the cardinality of the set $supp(\sigma)\cap \{1,2,\dots,r\}$.  Using this analogy, it would be reasonable to call $\Omega_1^{(r,d)}$ an $r$-tuplet hyper-star on $rd$ vertices.  See Figure \ref{fig1} for a picture in the case $r=3$ and $d=2$.
\end{remark} 

\begin{remark} 
It is likely that $X(\Omega_1^{(r,d)})$ is a contractible space. However, leaf-equivalence does not seem to be a suitable approach to prove this, as one can see in Remark \ref{rembetti}.
\end{remark}

\section*{Statements and Declarations}

The authors have no relevant financial or non-financial interests to disclose. Data sharing is not applicable to this article as no data sets were generated or analyzed during the current study.



\section*{Acknowledgment}
We thank  Abraham Orinda, Tanmoy Rudra, and Kadir Yucel for some discussions in the early stage of this project. We thank Sebastian Cioabă for some comments on the final version. 

\bibliographystyle{amsalpha}

\end{document}